\documentclass{amsproc}

\usepackage{amsfonts,amssymb,amsmath}

\newtheorem{theorem}{Theorem}[section]
\newtheorem{lemma}[theorem]{Lemma}
\newtheorem{proposition}{Proposition}[section]
\newtheorem{corollary}{Corollary}[section]

\theoremstyle{definition}

\theoremstyle{remark}
\newtheorem{remark}[theorem]{Remark}

\numberwithin{equation}{section}

\begin{document}

\title[Upper central series for the group of unitriangular automorphisms...]
{Upper central series for the group of unitriangular automorphisms of a free associative algebra}

\author{Valeriy G. Bardakov}
\address{Sobolev Institute of Mathematics, Novosibirsk State University, Novosibirsk 630090, Russia}
\address{and}
 \address{Laboratory of Quantum Topology, Chelyabinsk State University, Brat'ev Kashirinykh street 129,
 Chelyabinsk 454001, Russia}
\email{bardakov@math.nsc.ru}
\thanks{The first author thank the organizers of the Conference ``Groups, Geometry and Dynamics"
(Almora, India, 2012) for this beautiful and interesting Conference.}

\author{Mikhail V. Neshchadim}
\address{Sobolev Institute of Mathematics, Novosibirsk State University, Novosibirsk 630090, Russia}
\email{neshch@math.nsc.ru}
\thanks{The  authors was supported by Federal Target Grant ``Scientific
and educational personnel of innovation Russia'' for 2009-2013
(government contract No. 02.740.11.0429). Also, this research was
supported by  the Indo-Russian DST-RFBR project
grant DST/INT/RFBR/P-137 (No.~13-01-92697)}





\begin{abstract}
We study some subgroups of the group of unitriangular automorphisms $U_n$
of a free associative algebra  over a field of characteristic zero. We find the center
of  $U_n$ and describe the hypercenters of $U_2$ and $U_3$. In particular, we
prove that the upper central series for $U_2$ has infinite length.
As consequence,  we prove that the
groups $U_n$ are non-linear for all $n \geq 2$.
\end{abstract}

\maketitle








\section{Introduction}

In this paper we consider a free associative algebra $A_n = K
\langle x_1, x_2,$ $\ldots,$ $x_n \rangle$ over a field $K$ of
characteristic zero. We assume that $A_n$ has unity. The group of $K$-automorphisms of this
algebra, i.e. automorphisms that fix elements of $K$ is
denoted by $\mathrm{Aut} \, A_n$.
The group of tame automorphisms $\mathrm{TAut} \, A_n$ of $A_n$
is generated by the group of affine automorphism
$\mathrm{Aff} \, A_n$
and the group of unitraingular automorphisms $U_n=U(A_n)$.
From the result of Umirbaev \cite{U}  follows that $\mathrm{Aut} \, A_3 \not= \mathrm{TAut} \, A_3$.

A question about linearity (i.e. about a faithful representation
by finite dimensional matrices over some field) of $\mathrm{TAut}
\, A_n$ was studied in the paper of Roman'kov, Chirkov, Shevelin
\cite{R}, where it was proved that for $n \geq 4$ these groups are
not linear. Sosnovskii \cite{S} proved that for  $n \geq 3$ the
group $\mathrm{Aut} \, P_n$ is not linear, where
$P_n = K [x_1, x_2,$ $\ldots,$ $x_n ]$ is the polynomial algebra over a field $K$.
His result follows
from description of the upper central series for the group of
unitriangular automorphisms $U(P_n)$ of $P_n$.

The structure of the present paper is the following. In Section 2 we introduce  some notations, recall
some facts on the automorphism group of free associative algebra and its subgroups.
In the previous article \cite{B} we
found the lower central series and the series of the commutator subgroups for $U_n$.
In Section 3 we study the upper central series for $U_2$ and prove that the length of this series is infinite.
Prove that $U_2$ is non-linear. In Section 4 we study the upper central series for $U_3$ and describe the
hypercentral subgroups in the terms of some algebras. In Section 5 we find the center of $U_n$ for $n \geq 4$.
Also, in Sections 4 and 5 we formulate some hypotheses and questions
that are connected with the
theory of non-commutative invariants in free associative algebra under the action of some subgroups of $U_n$.

\section{Some previous results and remark}

Let us recall definitions of some automorphisms and subgroups of $\mathrm{Aut} \, A_n$.

For any index  $i\in \left\{1, \ldots, n \right\}$, a constant
$\alpha \in K^* = K\backslash \{0\}$ and a polynomial $f = f(x_1,
\ldots , \widehat{x_i}, \ldots ,x_n) \in A_n$ (where the symbol
$\widehat{x_i}$ denotes that $x_i$ is not included in $f$) {\it the
elementary automorphism } $\sigma (i, \alpha, f)$ is an
automorphism in $\mathrm{Aut} \, A_n$ that acts on the
variables  $x_1, \ldots ,x_n$ by the rule:
$$
\sigma (i, \alpha, f) :
 \left\{
\begin{array}{lcl}
x_i \longmapsto \alpha \, x_i + f, \,\, &    \\
x_j \longmapsto x_j, \,\,           & \mbox{if} & \,\, j \neq i. \\
\end{array}
\right.
$$
The group of tame automorphisms $\mathrm{TAut} \, A_n$ is
generated by all elementary automorphisms.

The group of affine automorphisms $\mathrm{Aff} \, A_n$ is a
subgroup of $\mathrm{TAut} \, A_n$ that consists of automorphisms
$$
x_i\longmapsto a_{i1} x_1+ \ldots + a_{in} x_n + b_i,\,\, i=1, \ldots , n,
$$
where $a_{ij}$, $b_i\in K$, $i,j=1, \ldots ,n$, and the matrix
$(a_{ij})$ is a non-degenerate one.  The group of affine
automorphisms is the semidirect product $K^n \leftthreetimes
\mathrm{GL}_n (K)$ and, in particular, embeds in the group of
matrices $\mathrm{GL}_{n+1} (K)$.

The group of triangular automorphisms  $T_n = T(A_n)$ of algebra
$A_n$  is generated by automorphisms
$$
x_i\longmapsto \alpha_i x_i+f_i(x_{i+1}, \ldots , x_n),\,\, i = 1,\ldots , n,
$$
where $\alpha_i \in K^*$, $f_i\in A_n$ and $f_n\in K$. If all
$\alpha_i =1$ then this automorphism is called {\it the
unitriangular automorphism}. The group of unitriangular
automorphisms is denoted by $U_n = U(A_n)$.

In the group $U_n$ let us define a subgroup $G_i$, $i = 1, 2, \ldots, n$ which is generated by automorphisms
$$
 \sigma (i, 1, f),~~ \, \mbox{where} \,  f = f(x_{i+1}, \ldots , x_n) \in A_n.
$$
Note that the subgroup $G_i$ is abelian and isomorphic to an
additive subgroup of algebra $A_n$ that is generated by $x_{i+1},
\ldots , x_n$, $i = 1, \ldots, n-1$, and the subgroup  $G_n $ is
isomorphic to the additive group of the field $K$.

 {\it The lower central series}
of a group $G$ is the series
$$
G = \gamma_1 G \geq \gamma_2 G \geq \ldots,
$$
where $\gamma_{i+1} G = [\gamma_i G, G],$ $i = 1, 2, \ldots$. {\it The series of the commutator subgroups} of a
group $G$ is the series
$$
G = G^{(0)} \geq G^{(1)} \geq G^{(2)} \geq \ldots,
$$
where $G^{(i+1)} = [G^{(i)}, G^{(i)}],$ $i = 0, 1, \ldots$. Here for subsets $H$, $K$ of $G$, $[H, K]$
denotes the subgroup of $G$ generated by the commutators $[h, k] = h^{-1} k^{-1} h k$ for $h \in H$ and
$k \in K$.

Recall that the $k$-th hypercenter $Z_k = Z_k(G)$ of the upper central series of $G$ for the
non-limited ordinal $k$ is defined by the rule
$$
Z_k / Z_{k-1} = Z(G / Z_{k-1})
$$
or equivalently,
$$
Z_k = \{ g \in G ~|~[g, h] \in Z_{k-1}~ \mbox{for all} ~h \in G \},
$$
and $Z_1(G) = Z(G)$ is the center of $G$.
If $\alpha$ is a limit ordinal, then define
$$
Z_{\alpha} = \bigcup_{\beta < \alpha} Z_{\beta}.
$$

It was proved in \cite{B} that $U_n$ is a semidirect product of abelian groups:
$$
U_n = (\ldots(G_1 \leftthreetimes G_2)\leftthreetimes  \ldots ) \leftthreetimes G_n,
$$
and  the lower central
series and the series of commutator subgroups of $U_n$ satisfy the following two properties respectvely:

1)  For $n \geq 2$
$$
\gamma_2 U_n = \gamma_3 U_n = \ldots.
$$
In particular, for $n \geq 2$ the group $U_n$ is not nilpotent.

2)  The group $U_n$ is solvable of degree $n$ and the
corresponding commutator subgroups have the form:
$$
 \begin{array}{l}
  U_n= (\ldots(G_1 \leftthreetimes G_2)\leftthreetimes  \ldots ) \leftthreetimes G_{n},\\
  U_n^{(1)}= (\ldots(G_1 \leftthreetimes G_2)\leftthreetimes  \ldots ) \leftthreetimes G_{n-1},\\
  .........................................\\
  U_n^{(n-1)}= G_1,\\
  U_n^{(n)}= 1.
 \end{array}
$$

Yu. V. Sosnovskiy \cite{S} found the upper central series for the unitriangular group $U(P_n)$ of the
polynomial algebra $P_n$. (Note, that he considered polynomials without free terms.)
He proved that for $n \geq 3$ the group $U(P_n)$ has the  upper central series
of length $((n-1)(n-2)/2) \omega + 1$ for  any field $K$, where $\omega$ is the first limit ordinal.
If $\mathrm{char} \, K = 0$ then the
hypercenters of $U(P_4)$ have the form\\

$Z_{k} = \{ (x_1 + f_1(x_3, x_4), x_2, x_3) ~|~
\mathrm{deg}_{x_3} f_1(x_3, x_4) \leq k -1 \}$,\\

$Z_{\omega} = \{ (x_1 + f_1(x_3, x_4), x_2, x_3)  \}$,\\

$Z_{\omega + k} = \{ (x_1 + f_1(x_2, x_3, x_4), x_2, x_3) ~|~
\mathrm{deg}_{x_2} f_1(x_2, x_3, x_4) \leq k \}$,\\

$Z_{2 \omega} = \{ (x_1 + f_1(x_2, x_3, x_4), x_2, x_3)  \}$,\\

$Z_{2\omega + k} = \{  (x_1 + f_1(x_2, x_3, x_4), x_2 + f_2(x_3, x_4), x_3) ~|~
\mathrm{deg}_{x_3} f_2(x_3, x_4) \leq k - 1 \}$,\\

$Z_{3 \omega} = \{  (x_1 + f_1(x_2, x_3, x_4), x_2 + f_2(x_3, x_4), x_3)  \}$,\\

$Z_{3 \omega + 1} = U(P_4)$,\\
where $k = 1, 2, \ldots$ runs over the set of natural numbers and
$f_1, f_2, f_3$ are arbitrary polynomials in $P_n$ which depend on the corresponding variables.

\section{Unitriangular group $U_2$}

Let $A_2 = K \langle x, y \rangle$ be the free associative   algebra over a field $K$  of
characteristic zero with the variables $x$ and $y$. Then
$$
U_2=
\left\{
 \varphi= \left( x + f(y), y + b  \right) ~|~ f(y) \in K\langle  y \rangle,\,\,b \in K \right\}
$$
is the group of unitriangular automorphisms of  $A_2$.

It is not difficult to check the following Lemma

\begin{lemma}\label{l:form}
1) If $\varphi = \left( x + f(y), y + b  \right) \in U_2$,
then its inverse is equal  to
$$
 \varphi^{-1}= \left( x - f(y - b), y - b  \right);
$$

2) if
$ \varphi= \left( x + f(y), y + b \right)$ and $\psi= \left( x + h(y), y + c \right)\in U_2$,
then the following formulas hold:

-- the formula of conjugation
$$
  \psi^{-1}\varphi \psi = \left( x + h(y) - h(y + b) + f(y + c), y + b  \right),
$$

-- the formula of commutation
$$
  \varphi^{-1}\psi^{-1}\varphi \psi = \left( x + h(y) - h(y + b) + f(y + c) - f(y), y  \right).
$$

\end{lemma}

Using this Lemma we can describe the center of $U_2$.

\begin{lemma} \label{l:c}
The center of $U_2$ has the form
$Z(U_2) = \left\{ \varphi = \left( x + a, y  \right) ~|~ a \in K \right\}$.
\end{lemma}

\begin{proof}
If $\varphi = \left( x + a, y  \right)$, then from the formula of conjugation (see Lemma~\ref{l:form}) follows
 that $\varphi \in Z(U_2)$. To prove the inverse inclusion
suppose
$ \varphi= \left( x+f(y), y+b  \right) \in Z(U_2)$.
Using the formula of conjugation we get
$$
  \varphi= \left( x + f(y), y + b  \right) = \psi^{-1} \varphi \psi = \left( x + h(y) - h(y + b) + f(y + c)
  y + b  \right),
$$
for any automorphism  $\psi = \left( x + h(y), y + c  \right) \in U_2$,
i.e.
$  f(y) = h(y) - h(y + b) + f(y + c)$.
Taking $h = 0$ we get $f(y) = f(y + c)$ for every $c \in K$.
Hence, $f(y) = a \in K$. We have only relation
$0 = h(y) - h(y + b)$. Since $h(y)$ is arbitrary, it follows that $b=0$.

\end{proof}

\begin{lemma}\label{l:com}
The following properties hold true in $U_2$.

1)
$[U_2, U_2] = \left\{\varphi = \left( x + f(y), y  \right) ~|~ f(y) \in K \langle  y \rangle \right\}$.

2) If $\varphi = \left( x + f(y), y  \right)$, where $f(y) \in K \langle  y \rangle \setminus K$, then
$$
C_{U_2}(\varphi) = \left\{\left( x + h(y), y  \right) ~|~ h(y)\in K \langle y \rangle \right\},
$$
where $C_{U_2}(\varphi)$ is the centralizer of $\varphi$ in $U_2$, i.e.
$C_{U_2}(\varphi) = \{ \psi \in U_2 ~|~ \psi \varphi = \varphi \psi\}$.

3) If $\varphi= \left( x,y+b  \right)$, $b \in K$, then
$C_{U_2}(\varphi) = \left\{ \left( x + a, y + c  \right) ~|~ a, c \in K \right\}$.

\end{lemma}

\begin{proof} 1) Let
$\varphi = \left( x + f(y), y  + b \right)$, $\psi = \left( x + h(y), y + c  \right) \in U_2$.
By the formula of commutation
$$
  \varphi^{-1} \psi^{-1} \varphi \psi = \left( x + h(y) - h(y + b) + f(y + c) - f(y), y  \right).
$$
It is easy to see that any element of $K \langle  y \rangle$ can be represented in a form
$r(y + d) - r(y)$ for some  $r(y) \in K \langle  y \rangle$ and $d\in K$.
Hence,
$$
K \langle  y \rangle = \{ h(y) - h(y + b) + f(y + c) - f(y)~|~h, f \in K \langle  y \rangle,~~b, c \in K\}
$$
and 1) is true.

2) Let
$ \varphi = \left( x+f(y), y \right)$, $f(y) \in K \langle y \rangle \setminus K$ and
$ \psi = \left( x + h(y), y + c  \right)$ be an arbitrary element of $C_{U_2}(\varphi)$.
Using the formula of conjugation we get
$$
  \varphi= \left( x + f(y), y  \right) = \psi^{-1} \varphi \psi = \left( x + f(y + c), y  \right).
$$
Hence $c = 0$.

3) Let
$ \varphi = \left( x, y + b \right)$ and
$ \psi = \left( x + h(y), y + c  \right)$ be an arbitrary element of $C_{U_2}(\varphi)$.
Using the formula of conjugation we get
$$
  \varphi = \left( x, y + b  \right) = \psi^{-1} \varphi \psi = \left( x + h(y) - h(y + b), y + b  \right).
$$
Hence $h(y)=a \in K$.

\end{proof}

\begin{lemma}\label{l:co}
If $s$ is a non-negative integer, then the $(s+1)th$ hypercenter of $U_2$ has the form
$$
 Z_{s+1}(U_2) = \left\{\varphi = \left( x+f(y),y  \right)
~|~ f(y) \in K \langle  y \rangle,\,\, \mathrm{deg} f(y) \leq s \right\}.
$$
\end{lemma}

\begin{proof}
If $s = 0$, then the assertion follows from Lemma \ref{l:c}. Suppose that for
$s + 1$ the assertion holds true. We now prove it for  $s + 2$.
Let
$$
  \varphi = \left( x + f(y), y + b  \right)\in Z_{s+2}(U_2).
$$
Using the formula of commutation (see Lemma \ref{l:form})
for $\varphi$ and $\psi= \left( x + h(y), y + c  \right)$
we get
$$
  \varphi^{-1}\psi^{-1}\varphi \psi = \left( x + h(y) - h(y + b) + f(y + c) - f(y), y  \right).
$$
If $b \neq 0$ and since $h(y)$ is an arbitrary  polynomial of $K \langle  y \rangle$, then
$h(y) - h(y + b) + f(y + c) - f(y)$ represents arbitrary element of
$K \langle y \rangle$.
But it is not possible since for any automorphism $\varphi = (x + f(y), y) \in Z_{s+1}$ the degree
of $f(y)$ is not bigger than $s$.
Hence $b = 0$. Since $\mathrm{deg} f(y+c)-f(y) \leq s$ and  $c$ is an arbitrary element of $K$,
we have $\mathrm{deg} f(y) \leq s + 1$. So the inclusion from left to right is proved. The inverse inclusion is
evident.

\end{proof}

Hence, from Lemma \ref{l:co}
$$
 Z_{\omega}(U_2)=\left\{\varphi= \left( x + f(y), y  \right) ~|~ f(y) \in K\langle y \rangle \right\},
$$
and using Lemma \ref{l:com} we have $Z_{\omega}(U_2)=[U_2, U_2]$.
Therefore
$$
 Z_{\omega+1}(U_2)=U_2.
$$

\begin{corollary}
The group $U_2$ is not linear.
\end{corollary}

\begin{proof}
We know (see, for example \cite{G}) that if a linear group does not
contain torsion, then the length of the upper central series of this group is finite.
But we proved that  the length of the upper central series for
$U_2$ is equal to $\omega + 1$. Hence, the group $U_2$ is not linear.
\end{proof}

Since $U_2$ is a subgroup of $\mathrm{Aut} \, A_2$,  we have proven
that this group is not linear too. Using the fact that if $P_2$ is a polynomial algebra with unit over
$K$, then $\mathrm{Aut} \, A_2 = \mathrm{Aut} \, P_2$
(see, for example \cite{C}).

\begin{corollary}
Let $n \geq 2$. Then the groups $\mathrm{Aut} \, A_n$ and $\mathrm{Aut} \, P_n$ are
not linear.
\end{corollary}

It  follows from the fact that $\mathrm{Aut} \, A_2 \leq \mathrm{Aut} \, A_n$
and $\mathrm{Aut} \, P_2 \leq \mathrm{Aut} \, P_n$ for all $n \geq 2$.

\begin{remark}
In \cite{S} the author considered the polynomials without free terms and proved that $\mathrm{Aut} \, P_3$
is not linear. Using his method it is not difficult to prove that if $P_2$ contains free terms,  then
$\mathrm{Aut} \, P_2$ is not linear over any  field of arbitrary characteristic.
\end{remark}

\section{Unitriangular group $U_3$}

The group $U_3$ is equal to
$$
U_3 = \{ (x_1 + f_1, x_2 + f_2, x_3 + f_3) ~|~f_1 = f_1(x_2, x_3) \in K \langle x_2, x_3 \rangle,
$$
$$
f_2 = f_2(x_3) \in K \langle x_3 \rangle, f_3 \in K \}.
$$

Define an algebra $S$ as subalgebra of $K \langle x_2, x_3 \rangle$
$$
S = \{ f(x_2, x_3)\in K \langle x_2, x_3 \rangle ~|~ f(x_2 + g(x_3), x_3 + h) = f (x_2, x_3)~
$$
$$
\mbox{for any}~g(x_3) \in K \langle x_3 \rangle,~~h \in K \}
$$
Hence, $S$ is a subalgebra of fixed elements under the action of the group
$$
\{ (x_2 + g, x_3 + h) ~|~g = g(x_3) \in K \langle x_3 \rangle, h \in K \}
$$
which is isomorphic to $U_2$.
The set $S$ is a subalgebra of $A_3$.

Define a set of commutators
$$
c_1 = [x_2, x_3],~~c_{k+1} = [c_k, x_3],~~k = 1, 2, \ldots,
$$
where $[a, b] = a b - b a$ is the ring commutator. Using induction on $k$, it is not difficult to check
the following result.

\begin{lemma}
The commutators $c_k$, $k = 1, 2, \ldots,$ lie in $S$.
\end{lemma}

{\bf Hypothesis 1.} The algebra $K \langle c_1, c_2, \ldots \rangle$ is equal to $S$.

\medskip

Note that the elements $c_1, c_2, \ldots$ are free generators of $K \langle c_1, c_2, \ldots \rangle$
(see \cite[p.~62]{Co}).

\medskip

\begin{theorem} The center $Z(U_3)$ of the group $U_3$ is equal to
$$
Z(U_3) = \{ (x_1 + f_1, x_2, x_3) ~|~f_1 = f_1(x_2, x_3) \in S \}.
$$
\end{theorem}

\begin{proof} The inclusion $\supseteq$ is evident.  Let
$$
\varphi = (x_1 + f_1(x_2, x_3), x_2 + f_2(x_3), x_3 + f_3)
$$
be some element in $Z(U_3)$ and
$$
\psi = (x_1 + g_1(x_2, x_3), x_2 + g_2(x_3), x_3 + g_3)
$$
be an arbitrary element of $U_3$. Since $\varphi \psi = \psi \varphi$ then we have the equalities\\

$x_1 + g_1(x_2, x_3) + f_1(x_2 + g_2(x_3), x_3 + g_3) =
x_1 + f_1(x_2, x_3) + g_1(x_2 + f_2(x_3), x_3 + f_3),$\\

$x_2 + g_2(x_3) + f_2(x_3 + g_3) = x_2 + f_2(x_3) + g_2(x_3 + f_3),$\\

$x_3 + g_3 + f_3 = x_3 + f_3 + g_3.$ \\

\noindent The third equality  holds for all $x_3$, $g_3$ and $f_3$. Rewrite the first and the second equality
in the form\\

$g_1(x_2, x_3) + f_1(x_2 + g_2(x_3), x_3 + g_3) = f_1(x_2, x_3) + g_1(x_2 + f_2(x_3), x_3 + f_3),$\\

$g_2(x_3) + f_2(x_3 + g_3) = f_2(x_3) + g_2(x_3 + f_3).$\\

\noindent Let $g_1 = x_2$, $g_2 = g_3 = 0$. Then
$$
x_2 + f_1(x_2, x_3) = f_1(x_2, x_3) + x_2 + f_2(x_3).
$$
Hence $f_2(x_3) = 0$.

Let $g_1 = x_3$, $g_2 = g_3 = 0$. Then
$$
x_3 + f_1(x_2, x_3) = f_1(x_2, x_3) + x_3 + f_3.
$$
Hence $f_3 = 0.$
We have only one condition
$$
f_1(x_2 + g_2(x_3), x_3 + g_3) = f_1(x_2, x_3),
$$
i.e. $f_1 \in S$.

\end{proof}


Let us  define the following subsets in the algebra $A_2 = K\langle x_2, x_3 \rangle$:\\

$S_1 = S$,\\

$S_{m+1} = \{ f \in A_2 ~|~f^{\varphi} - f \in S_m~ \mbox{for all}~ \varphi \in U_2 \}$, $m = 1, 2, \ldots$,\\

$S_{\omega} = \bigcup\limits_{m = 1}^{\infty} S_{m}$, \\

$S_{\omega+1} = \{ f \in A_2 ~|~f^{\varphi} - f \in S_\omega~ \mbox{for all}~ \varphi \in U_2 \}$,\\

$S_{\omega+m+1} = \{ f \in A_2 ~|~f^{\varphi} - f \in S_{\omega+m}~ \mbox{for all}~ \varphi \in U_2 \}$,
$m = 1, 2, \ldots$,\\

$S_{2\omega} = \bigcup\limits_{m = 1}^{\infty} S_{\omega+m}$, \\

$R_{m} = \{ f = f(x_3) \in K \langle x_3 \rangle ~|~\mathrm{deg} \, f \leq m \}$, $m = 0, 1, \ldots ,$\\

$R_{\omega} = \bigcup\limits_{m = 0}^{\infty} R_{m}$. \\
It  is not difficult to see that all $S_k$  are modules over $S$.

\medskip

\begin{remark} If we consider the homomorphism
$$
\pi : K \langle x_1, x_2, x_3 \rangle \longrightarrow K [x_1, x_2, x_3],
$$
then\\

$S^{\pi} = K$,\\

$S_{m+1}^{\pi} = \{ f \in K[x_3] ~|~\mathrm{deg} \, f \leq m \}$, $m = 1, 2, \ldots$,\\

$S_{\omega}^{\pi} = K[x_3]$, \\

$S_{\omega+m+1}^{\pi} = \{ f \in K[x_2, x_3] ~|~\mathrm{deg}_{x_2} \, f \leq m \}$,
$m = 0, 1, \ldots$,\\

$S_{2\omega}^{\pi} = K[x_2, x_3]$, \\

$R_{m}^{\pi} = \{ f \in K[x_3] ~|~\mathrm{deg} \, f \leq m \}$, $m = 1, 2, \ldots ,$\\

$R_{\omega}^{\pi} = K[x_3]$. \\

\end{remark}

\begin{theorem} The following equalities hold
\begin{equation}\label{eq:m}
Z_{m} = \{ (x_1 + f_1(x_2, x_3), x_2, x_3)~|~ f_1 \in S_m \},~~ m = 1, 2, \ldots,  2\omega,
\end{equation}
\begin{multline}\label{eq:2m}
Z_{2\omega+m} = \{ (x_1 + f_1(x_2, x_3), x_2 + f_2(x_3), x_3)~|~ f_1 \in k\langle x_2, x_3 \rangle,
f_2(x_3) \in R_m \},\\  m = 1, 2, \ldots,  \omega,
\end{multline}
\begin{equation}\label{eq:3m}
Z_{3\omega + 1} = U_3.
\end{equation}
\end{theorem}

\begin{proof}  We use  induction on $m$. To prove (\ref{eq:m}) for $m+1$, we assume that for all $m$ such that
$1 \leq m < \omega$ equality (\ref{eq:m}) holds.
If
$$
\varphi = (x_1 + f_1(x_2, x_3), x_2 + f_2(x_3), x_3 + f_3) \in Z_{m+1}
$$
and
$$
\psi = (x_1 + g_1(x_2, x_3), x_2 + g_2(x_3), x_3 + g_3) \in U_{3},
$$
then for some
$$
\theta = (x_1 + h_1(x_2, x_3), x_2, x_3) \in Z_{m}
$$
holds $\varphi \, \psi = \psi \, \varphi \, \theta$. Acting on the generators $x_1$,  $x_2$,  $x_3$ by
$\varphi \, \psi$ and $\psi \, \varphi \, \theta$
we have
two relations
\begin{multline}\label{eq:1}
f_1(x_2 + g_2(x_3), x_3 + g_3) - f_1(x_2, x_3)  =  h_1(x_2, x_3) +  \\
+ g_1(x_2 + f_2(x_3), x_3 + f_3) - g_1(x_2, x_3),
\end{multline}
\begin{equation}\label{eq:2}
g_2(x_3) + f_2(x_2 + g_3) = f_2(x_3) + g_2(x_2 + f_3).
\end{equation}

If $g_2 = 0$, then the relation (\ref{eq:2}) has the form
$$
f_2(x_2 + g_3) = f_2(x_3).
$$
Since $g_3$ is an arbitrary element of $K$, then $f_2 \in K$. But in this case (\ref{eq:2}) has the form
$$
g_2(x_3 + f_3) = g_2(x_3).
$$
Since $g_2(x_3)$ is an arbitrary element of $K\langle x_3 \rangle$,  then $f_3 = 0$ and (\ref{eq:1}) has the
form\\
\begin{equation}\label{eq:3}
f_1(x_2 + g_2(x_3), x_3 + g_3) - f_1(x_2, x_3) = h_1(x_2, x_3) + g_1(x_2 + f_2, x_3) - g_1(x_2, x_3).
\end{equation}\\
Let $g_1 = x_2^{N}$ for some natural number $N$. Using the homomorphism
$$
\pi : K \langle x_1, x_2, x_3 \rangle \longrightarrow K[x_1, x_2, x_3]
$$
and the equality $\mathrm{deg}_{x_2}\, h_1^{\pi} = 0$ we see that if $f_2 \not= 0$, then
$$
\mathrm{deg}_{x_2}\, \left( f_1(x_2 + g_2(x_3), x_3 + g_3) - f_1(x_2, x_3) \right)^{\pi} = N - 1.
$$
Since $N$ is any non-negative integer,  then $f_2 = 0$ and
$$
f_1(x_2 + g_2(x_3), x_3 + g_3) - f_1(x_2, x_3) \in S_m,
$$
i.e. $f_1(x_2, x_3) \in S_{m+1}$ and we have proven the equality (\ref{eq:m}) for $m+1$:
$$
Z_{m+1} = \{ (x_1 + f_1(x_2, x_3), x_2, x_3)~|~ f_1 \in S_{m+1} \}.
$$

To prove (\ref{eq:m}) for $\omega+m+1$ assume that for all $\omega+m$ such that
$1 \leq m < \omega$ equality (\ref{eq:m}) holds.
If $\varphi \in Z_{\omega+m+1}$ and $\psi \in U_3$ then for some $\theta \in Z_{\omega+m}$ we have
$\varphi \, \psi = \psi\, \varphi \, \theta$ that give the relations (\ref{eq:1}) and (\ref{eq:2}).
As in the previous case we
can check that $f_2 \in K$, $f_3 = 0$ and (\ref{eq:1})--(\ref{eq:2}) are equivalent to (\ref{eq:3}).
Let $g_1 = x_2^{N}$ for some natural number $N$. Using the homomorphism
$$
\pi : K \langle x_1, x_2, x_3 \rangle \longrightarrow K [x_1, x_2, x_3]
$$
and the inequality $\mathrm{deg}_{x_2}\, h_1^{\pi} \leq m - 1$ we see that if $f_2 \not= 0$, then
$$
\mathrm{deg}_{x_2}\, \left( f_1(x_2 + g_2(x_3), x_3 + g_3) - f_1(x_2, x_3) \right)^{\pi} = N - 1
$$
for  $N \geq m + 1$. But the degree of the left hand side is bounded.
Hence $f_2 = 0$ and we have
$$
f_1(x_2 + g_2(x_3), x_3 + g_3) - f_1(x_2, x_3) \in S_{\omega+m},
$$
i.e.,  $f_1(x_2, x_3) \in S_{m+1}$ and  we have proven the equality (\ref{eq:m}) for $\omega+m+1$:
$$
Z_{\omega+m+1} = \{ (x_1 + f_1(x_2, x_3), x_2, x_3)~|~ f_1 \in S_{\omega+m+1} \}.
$$

To prove (\ref{eq:2m}) for $m+1$ assume that for all $m$ such that
$1 \leq m < \omega$ equality (\ref{eq:2m}) holds.
If $\varphi \in Z_{2\omega+m+1}$, $\psi \in U_3$, then for some $\theta \in Z_{2\omega+m}$ we have
$\varphi \, \psi = \psi\, \varphi \, \theta$. If
$$
\varphi = (x_1 + f_1(x_2, x_3), x_2 + f_2(x_3), x_3 + f_3),
$$
$$
\psi = (x_1 + g_1(x_2, x_3), x_2 + g_2(x_3), x_3 + g_3)
$$
and
$$
\theta = (x_1 + h_1(x_2, x_3), x_2 + h_2(x_3), x_3),
$$
then we have the relations

\begin{equation*}
\begin{split}
x_1 + g_1(x_2, x_3) + f_1(x_2 + g_2(x_3), x_3 + g_3) & = x_1 + h_1(x_2, x_3) + f_1(x_2 + h_2(x_3), x_3) +\\
& + g_1(x_2 + h_2(x_3) + f_2(x_3), x_3 + f_3),
\end{split}
\end{equation*}

$$
x_2 + g_2(x_3) + f_2(x_3 + g_3) = x_2 + h_2(x_3) + f_2(x_3)  + g_2(x_3 + f_3).
$$
Since $h_1$ is an arbitrary element of $K \langle x_2, x_3 \rangle$, then we must consider only the second
relation which is equal to
$$
f_2(x_3 + g_3) - f_2(x_3) = h_2(x_3) + g_2(x_3 + f_3) - g_2(x_3).
$$
Since $\mathrm{deg} \, h_2 \leq m$ and $g_2(x_3)$ is any element of $K \langle x_3 \rangle$, then $f_3 = 0$.
Hence,
$$
f_2(x_3 + g_3) - f_2(x_3) = h_2(x_3).
$$
From this equality follows that $\mathrm{deg} \, f_2 \leq m + 1$. We have proven  the equality (\ref{eq:2m}) for $m+1$:
$$
Z_{2\omega+m+1} = \{ (x_1 + f_1(x_2, x_3), x_2 + f_2(x_3), x_3)~|~ f_1 \in K \langle x_2, x_3 \rangle,
f_2(x_3) \in R_{m+1} \}.
$$

To prove (\ref{eq:2m}) we note that
$$
[U_3, U_3] \subseteq \{ (x_1 + f_1(x_2, x_3), x_2 + f_2(x_3), x_3) \}.
$$

\end{proof}

We described  the hypercenters of $U_3$ in the terms of the algebras $S_m$ and  $S_{\omega+m}$. It is
interesting to find sets of generators for these algebras.
To do it we must give answers on the following questions.

{\bf Question 1} (see Hypothesis 1). Is it true that
$$
S = K\langle c_1, c_2, \ldots \rangle?
$$

\medskip

{\bf Question 2.} Is it true that for all $m \geq 1$ the following equalities are true
$$
S_{m+1} = \{ f \in K\langle S, x_3 \rangle ~|~\mathrm{deg}_{x_3} \, f \leq m \}?
$$

\medskip

{\bf Question 3.} Is it true that
$$
\bigcup_{m=1}^{\infty} S_m = K\langle S, x_3 \rangle?
$$

\medskip

{\bf Question 4.} Is it true that for all $m \geq 1$ the following equalities are true
$$
S_{\omega+m} = \{ f \in K \langle S, x_3, x_2 \rangle ~|~\mathrm{deg}_{x_2} \, f \leq m \}?
$$

\medskip

If $R$ is the Specht algebra of $A_2$, i.e. the subalgebra of $A_2$ that is generated by all commutators
$$
[x_2, x_3],~~[[x_2, x_3], x_3],~~[[x_2, x_3], x_2], \ldots
$$
then the following inclusions hold

\begin{equation}\label{eq:4}
S_{m+1} \subset \{ f \in K \langle R, x_3 \rangle ~|~\mathrm{deg}_{x_3} \, f \leq m \},
\end{equation}

\begin{equation}\label{eq:5}
S_{\omega+m} \subset \{ f \in K \langle R, x_3, x_2 \rangle ~|~\mathrm{deg}_{x_2} \, f \leq m \},
\end{equation}
for all $m \geq 1$. It  follows from the fact that $K \langle x_2, x_3 \rangle$ is a free left $R$-module
 with the set of free generators
$$
x_2^{\alpha} x_3^{\beta}, ~~~\alpha, \beta \geq 0.
$$
Note that the inclusion (\ref{eq:4})  is strict. It follows from

\begin{proposition}
The commutators
$$
[x_2, \underbrace{x_3, \ldots, x_3}_k, x_2] = [c_k, x_2],~~k \geq 1,
$$
do not lie in $S_{m}$, $m \geq 1$.
\end{proposition}

\begin{proof}
Indeed, for the automorphism
$$
\varphi = (x_2 + g_2(x_3), x_3 + g_3)
$$
of the algebra $K \langle x_2, x_3 \rangle$
we have
$$
[c_k, x_2]^{\varphi} = [c_k^{\varphi}, x_2^{\varphi}] = [c_k, x_2 + g_2(x_3)] = [c_k, x_2] + [c_k, g_2(x_3)].
$$
If $g_2(x_3) = x_3^{N}$, then\\
\begin{align*}
[c_k, x_3^{N}] & = c_k \, x_3^{N} -  x_3^{N} c_k = (c_k x_3 - x_3 c_k) x_3^{N-1} + x_3 c_k x_3^{N-1} -
x_3^{N} c_k \\
 & = c_{k+1} x_3^{N-1} + x_3 (c_k x_3^{N-1} - x_3^{N-1} c_k) \\
 & =c_{k+1} x_3^{N-1} +
x_3 \left(( c_k x_3 - x_3 c_k) x_3^{N-2} + x_3 c_k x_3^{N-2} - x_3^{N-1} c_k \right)\\
 & =c_{k+1} x_3^{N-1} + x_3 c_{k+1} x_3^{N-2} + x_3^{2} ( c_k x_3^{N-2} - x_3^{N-2} c_k )\\
 & = \ldots \\
 & = \sum_{p+q = N-1} x_3^p c_{k+1} x_3^q.\\
\end{align*}
Hence, for $\varphi = (x_2 + g_2(x_3), x_3 + g_3)$ we have
$$
[c_k, x_2]^{\varphi} - [c_k, x_2] = \sum_{p+q = N-1} x_3^p c_{k+1} x_3^q.
$$
If
$$
g_2(x_3) = \sum_{n=0}^{N} a_n x_3^n,
$$
then
$$
[c_k, x_2]^{\varphi} - [c_k, x_2] = \sum_{n=1}^{N-1} a_n \sum_{p+q=n-1} x_3^p c_{k+1} x_3^q.
$$
If $[c_k, x_2] \in S_m$ for some $m$, then
$$
[[c_k, x_2], \varphi ] \equiv [c_k, x_2]^{\varphi} - [c_k, x_2] \in S_{m-1}.
$$
Let
$$
\psi = (x_2 + h_2(x_3), x_3 + h),~~~\varphi = (x_2 + x_3^N, x_3).
$$
Then
\begin{align*}
[[[c_k, x_2], \varphi ], \psi ] & = \left[\sum_{p+q=N-1} x_3^p c_{k+1} x_3^q, \psi \right] \\
 & = \sum_{p+q=N-1} (x_3 + h)^p c_{k+1} (x_3 + h)^q - \sum_{p+q=N-1} x_3^p c_{k+1} x_3^q \\
 & = \sum_{p+q=N-1} \left( \sum_{l=0}^{p-1} C_p^l x_3^l h^{p-l} \right) c_{k+1}
 \left( \sum_{r=0}^{q} C_{q-1}^r x_3^r h^{q-r} \right)\\
\end{align*}
has the degree $N-2$ on the variable $x_3$. Continuing this process  we are getting that
if $\mathrm{deg} \, g_2(x_3) = N$, then
$$
[c_k, x_2, \varphi, \psi_1, \ldots, \psi_{N-1}] \in S
$$
for every  $\psi_1, \ldots, \psi_{N-1} \in U_2$. Hence,\\

$[c_k, x_2, \varphi, \psi_1, \ldots, \psi_{N-1}] \in S,$\\

$...............................$ \\

$[[c_k, x_2, \varphi ] \in S_N,$\\

$[c_k, x_2] \in S_{N+1}.$\\

Using the similar ideas we can prove that
$$
[c_k, x_2] \not\in S_{N}.
$$

Since we can take arbitrary number  $N > m$, then
$$
[c_k, x_2] \not\in S_{m}, ~~m = 1, 2, \ldots
$$
\end{proof}

\section{Center of the unitriangular group $U_n$, $n \geq 4$}

In this section we prove the following assertion

\begin{theorem}  Any automorphism $\varphi$ in the center $Z(U_n)$  of $U_n$ has a form
$$
 \varphi=
 \left( x_1 + f(x_2, \ldots, x_n), x_2, \ldots, x_n  \right),
$$
where the polynomial $f$ is such that
$$
f(x_2 + g_2, \ldots, x_n + g_n) = f(x_2, \ldots, x_n)
$$
for every
$ g_2 \in K \langle x_3, \ldots, x_n \rangle$, $g_3\in K \langle x_4, \ldots, x_n \rangle, \ldots, g_n \in K$.

\end{theorem}

We will assume that  $U_{n-1}$ includes in $U_n$ by the rule
$$
U_{n-1} = \left\{ \varphi =  \left( x_1, x_2 + g_2, \ldots, x_n + g_n  \right) \in U_n ~|~
 g_2 \in K \langle x_3, \ldots, x_n \rangle,\,\, \ldots,  g_n \in K
\right\}.
$$
Hence we have the following sequence of inclusions for the subgroups $U_k$, $k=3,\ldots,n$
$$
U_n \geq U_{n-1} \geq \ldots  \geq U_{3}.
$$

In this assumption  we can formulate Theorem by the following manner
$$
Z(U_n) = \left\{
\varphi= \left( x_1 + f(x_2, \ldots, x_n), x_2, \ldots, x_n  \right) ~|~
f^{U_{n-1}}=f
\right\},
$$
where
$$
f^{U_{n-1}} = \{ f^{\psi} ~|~ \psi \in U_{n-1}   \}.
$$

\begin{proof}
Let
$$
\varphi= \left( x_1 + f_1, x_2 + f_2, \ldots, x_n + f_n  \right) \in Z(U_n)
$$
and
$$
\psi= \left( x_1 + g_1, x_2 + g_2, \ldots, x_n + g_n  \right)
$$
be an arbitrary element of  $U_n$.
Then $x_k^{\varphi\psi} = x_k^{\psi\varphi}$  for all $k = 1, 2, \ldots, n$.
In particular, if $k = 1$, then

\begin{equation}\label{eq:11}
( x_1 + f_1)^\psi = ( x_1 + g_1)^\varphi.
\end{equation}

Put $g_1 = x_2$, $g_2 = g_3 = \ldots = 0$.
Then this relation has the form
$$
 x_1 + x_2 + f_1 = x_1 + f_1 + x_2 + f_2.
$$
Hence, $f_2 = 0$.

Analogously, putting $g_1 = x_3$, $g_2 = g_3 = \ldots = 0$, we get $f_3 = 0$.
Hence, $f_2 = f_3 = \ldots = f_n = 0$.

The relation (\ref{eq:11}) for arbitrary $\psi$ has the form
$$
x + g_1(x_2, \ldots, x_n) + f_1( x_2 + g_2, \ldots,  x_n + g_n) = x + f_1( x_2, \ldots, x_n)
+ g_1(x_2, \ldots, x_n).
$$
Hence,
$$
f_1( x_2+g_2, \ldots, x_n + g_n) = f_1( x_2, \ldots, x_n).
$$

\end{proof}

\vspace{0.5cm}

Let us define the notations
$$
\zeta U_n = \left\{ f(x_2, \ldots, x_n) \in K \langle x_2, \ldots, x_n \rangle  ~|~ f^{U_{n-1}} = f \right\},
$$
$$
\zeta U_{n-1} = \left\{ f(x_3, \ldots, x_n) \in K \langle x_3, \ldots, x_n \rangle  ~|~ f^{U_{n-2}} = f \right\},
$$
$$
......................................................................................
$$
$$
\zeta U_{3} = \left\{ f(x_{n-2}, x_{n-1}, x_n) \in K \langle x_{n-2},x_{n-1},x_n \rangle  ~|~
f^{U_{2}} = f \right\}.
$$
Note that $\zeta U_{3} = S$.

We formulate the next hypothesis on the structure of the algebras $\zeta U_k$, $k=3, \ldots, n$.

\vspace{0.5cm}

{\bf Hypothesis  2.} The following inclusions hold
$$
\zeta U_4 \subseteq K \langle \zeta U_3, x_{n-2} \rangle,
$$
$$
\zeta U_5\subseteq K \langle \zeta U_4, x_{n-3} \rangle,
$$
$$
.................................,
$$
$$
\zeta U_n\subseteq K \langle \zeta U_{n-1}, x_{2} \rangle.
$$

Recall that by Hypothesis 1 we have
$$
\zeta U_3 = K \langle c_1, c_2, \ldots  \rangle,
$$
where $c_1 = [x_{n-1}, x_n]$, $c_{k+1} = [c_k, x_n]$, $k = 1, 2, \ldots $

\begin{proposition}
If Hypotheses 1 and 2 are true,  then  the following equality holds
$$
\zeta U_k = K \langle c_1, c_2, \ldots  \rangle,\,\,\,k = 3, 4, \ldots, n.
$$
\end{proposition}

\begin{proof}
For $k = 4$ Hypothesis 2 has the form
$$
\zeta U_4 \subseteq K \langle \zeta U_3, x_{n-2} \rangle,
$$
i.e.,  every polynomial  $f\in \zeta U_4$
can be represented in the form
$$
f = F(x_{n-2}, c_1, c_2, \ldots, c_N)
$$
for some non-negative integer $N$.
Applying the automorphism
$$
\psi = \left( x_1, x_2, \ldots,  x_{n-2} + g_{n-2} , x_{n-1}, x_n \right),
$$
we get
$$
F(x_{n-2} + g_{n-2}, c_1, c_2, \ldots, c_N) = F(x_{n-2}, c_1, c_2, \ldots, c_N).
$$
Here $g_{n-2} = g_{n-2}(x_{n-1}, x_n)$ is an arbitrary element of $K \langle x_{n-1}, x_n \rangle$.
Putting in this equality $g_{n-2} = c_{N+1}$ and $x_{n-2} = 0$ we have
$$
F(c_{N+1}, c_1, c_2, \ldots, c_N) = F(0, c_1, c_2, \ldots, c_N).
$$
Since $c_1, c_2, \ldots $ are free generators,
 $F$ does not contain the variable  $x_{n-2}$. Hence
$$
\zeta U_4 = K \langle c_1, c_2, \ldots \rangle.
$$
Analogously, we can prove the equality
$$
\zeta U_k = K \langle c_1, c_2, \ldots \rangle,\,\,\,k=3, 4, \ldots, n.
$$
\end{proof}

We see that the description of the hypercenters of $U_n$ (see Hypotheses 1 and 2) is connected with the
theory of non-commutative invariants in free associative algebra under the action of some subgroups of $U_n$.
We will study these invariants in next papers.

\vspace{0.8cm}

\vspace{1cm}

\bibliographystyle{amsalpha}

\end{document}